\documentclass[12pt]{amsart}

\usepackage{amssymb}
\usepackage{amsmath}
\usepackage{bm}
\usepackage{graphicx}
\usepackage{psfrag}
\usepackage{color}
\usepackage{url}
\usepackage{algpseudocode}
\usepackage{fancyhdr}
\usepackage{xy}
\usepackage{mathtools}
\usepackage[linesnumbered, ruled, vlined]{algorithm2e}
\usepackage{dirtytalk}
\usepackage{float}
\usepackage{stmaryrd}
\usepackage{hyperref}
\usepackage{cite}

\usepackage{array}
\usepackage{longtable}
\usepackage[table]{xcolor}
\input xy
\xyoption{all}

\definecolor{bleudefrance}{RGB}{205,205,250}
\definecolor{blizzardblue}{RGB}{235,235,250}

\newtheorem*{maintheorem*}{Main Theorem}
\newtheorem{theorem}{Theorem}[section]
\newtheorem{prop}[theorem]{Proposition}

\newtheorem{lemma}[theorem]{Lemma}
\newtheorem{cor}[theorem]{Corollary}
\newtheorem{question}[theorem]{Question}

\theoremstyle{definition}
\newtheorem{definition}[theorem]{Definition}
\newtheorem{remark}[theorem]{Remark}
\newtheorem{example}[theorem]{Example}

\numberwithin{equation}{section}

\newcommand{\nn}{\mathbb{N}}
\newcommand{\qq}{\mathbb{Q}}
\newcommand{\rr}{\mathbb{R}}
\newcommand{\zz}{\mathbb{Z}}

\newcommand{\betti}{\text{Betti}}

\hypersetup{
	pdftitle={FD},
	colorlinks=true,
	linkcolor=cyan,
	citecolor=cyan,
	urlcolor=wine
}


\begin{document}
	
	\mbox{}
	\title{Factorization Invariants of the Additive Structure of Exponential Puiseux Semirings}
	
	\author{Harold Polo}
	\address{Department of Mathematics\\University of Florida\\Gainesville, FL 32611, USA}
	\email{haroldpolo@ufl.edu}
	
	\date{\today}
	
	\begin{abstract}
		Exponential Puiseux semirings are additive submonoids of $\qq_{\geq 0}$ generated by almost all of the nonnegative powers of a positive rational number, and they are natural generalizations of rational cyclic semirings. In this paper, we investigate some of the factorization invariants of exponential Puiseux semirings and briefly explore the connections of these properties with semigroup-theoretical invariants. Specifically, we provide exact formulas to compute the catenary degrees of these monoids and show that minima and maxima of their sets of distances are always attained at Betti elements. Additionally, we prove that sets of lengths of atomic exponential Puiseux semirings are almost arithmetic progressions with a common bound, while unions of sets of lengths are arithmetic progressions. We conclude by providing various characterizations of the atomic exponential Puiseux semirings with finite omega functions; in particular, we completely describe them in terms of their presentations.
	\end{abstract}

\maketitle

%%%%%%%%%%%%%%%%%%%%%%%%
\section{Introduction} \label{sec:intro}

In a unique factorization domain (or UFD) we can write a nonzero nonunit element as a product of finitely many atoms (i.e., irreducibles) and such a representation is unique up to order and units. By relaxing this property, we obtain two larger classes of integral domains that we call atomic and half-factorial following Cohn~\cite{cohn} and Zaks~\cite{Zaks1980}, respectively. An integral domain is atomic provided that every nonzero nonunit factors into atoms, while an atomic domain is half-factorial (or HFD) if any two factorizations of a nonunit element have the same length. Factorization theory studies how far is an atomic domain from being either a UFD or an HFD, and several invariants have been introduced to quantify this deviation (see, for example, \cite{ChWWS1990,aG1997,aGGl1990,Valenza}). In 1992, Halter-Koch~\cite{halter-koch1992} expanded the scope of factorization theory to the more general (and still suitable) class of cancellative and commutative monoids. 

Puiseux monoids (i.e., additive submonoids of $\qq_{\geq 0}$) are natural generalizations of numerical monoids, and they have been used as an important source of examples in factorization theory. For instance, Coykendall and Gotti~\cite{JCFG2019} utilized Puiseux monoids to partially answer a question posed by Gilmer almost forty years ago in \cite[page~$189$]{Gilmer}, while Geroldinger et al. \cite{AGFGST2019} found, in the class of Puiseux monoids, the first example of an atomic primary monoid with irrational elasticity (\cite[Example~4.3]{AGFGST2019}). Computations of factorization invariants are highly tractable in the context of finitely generated monoids (see \cite{GS2016} and references therein); however, computations are usually arduous for their non-finitely generated counterparts. For example, completely describing the sets of lengths of certain Puiseux monoids is as hard as solving the Goldbach's conjecture (\cite[Section~6]{fG2019}).

Some additive submonoids of $\rr_{\geq 0}$ have dyadic monoidal structures in the sense that they contain $1$, the multiplicative identity, and are also closed under multiplication. These monoids, called \emph{positive semirings}, have motivated much interest lately. In~\cite{JCFG}, the authors studied atomic properties of the additive structure of the positive semiring $\nn_0[\alpha] = \{f(\alpha) \mid f(X) \in \nn_0[X]\}$, where $\alpha$ is a nonzero real number, while the elasticity and delta set of $\nn_0[\tau]$, where $\tau$ is a quadratic integer, were explored in \cite{CCMS09}. Furthermore, Baeth et al.~\cite{NBSCFG} investigated the dual nature of atomicity in the context of positive semirings using a methodology introduced in ~\cite{DDDFMZ1990}. Finally, the factorization invariants of the additive structure of rational cyclic semirings (i.e., Puiseux monoids generated by the nonnegative powers of a positive rational number) were thoroughly described in~\cite{CGG19}, while the arithmetic of their multiplicative structure was considered in \cite{NBFG2020}.

The class of atomic rational cyclic semirings  is one of the rare classes of non-finitely generated monoids for which we can provide detailed descriptions for many of their factorization invariants, and various generalizations of this class (e.g., \cite{SAJBHP2020,JCFG, HP2020}) have been scrutinized to gain insight about the atomic structure and factorization invariants of positive monoids (i.e., additive submonoids of the nonnegative cone of the real line). In \cite{SAJBHP2020}, Albizu-Campos et al. studied the atomic properties of \emph{exponential Puiseux semirings}, that is, Puiseux monoids generated by almost all of the nonnegative powers of a positive rational number. The purpose of the present article is to show that atomic exponential Puiseux semirings retain some of the nice factorization properties of atomic rational cyclic semirings.

We start by introducing the necessary background and notation to follow our exposition. In Section~$3$, we describe the $\mathcal{R}$-classes of atomic exponential Puiseux semirings, which allows us to compute their catenary degrees. Additionally, we show that the minimum and maximum of $\Delta(M)$ are attained at Betti elements for every atomic exponential Puiseux semiring $M$. Then, in Section $4$, we investigate the sets of lengths of these monoids. To be precise, we prove that if $M$ is a nontrivial atomic exponential Puiseux semiring then there exists $B \in \nn$ such that every $L \in \mathcal{L}(M)$ is an AAP with difference $\min\Delta(M)$ and bound $B$, while for each $k \in \nn^{\bullet}$ the set $\mathcal{U}_k(M)$ is an arithmetic sequence with difference $\min\Delta(M)$. We conclude by offering, in the last section, several characterizations of atomic exponential Puiseux semirings with finite omega functions; in particular, we completely describe them in terms of their presentations.

%%%%%%%%%%%%%%%%%%%%%%%%%%%%%
\section{Fundamentals} \label{sec:background}

We now present the concepts and notation related to our exposition. Reference material on non-unique factorization theory can be found in the monograph~\cite{GH06b} by Geroldinger and Halter-Koch. 

\subsection{Notation}

Let $\nn$ denote the set of nonnegative integers, and let $\mathbb{P}$ denote the set of prime numbers. Additionally, if $X$ is a subset of the rational numbers then we set $\overline{X} \coloneqq X \cup \{\infty\}$ and $X_{<q} \coloneqq \{x \in X \mid 0 \leq x < q\}$; we define $X_{\leq q}$, $X_{>q}$, and $X_{\geq q}$ in a similar way. For a positive rational number $r = n/d$ with $n$ and $d$ relatively prime positive integers, we call $n$ the \emph{numerator} and $d$ the \emph{denominator} of $r$, and we set $\mathsf{n}(r) := n$ and $\mathsf{d}(r) := d$. For nonnegative integers $k$ and $m$, we denote by $\llbracket k,m \rrbracket$ the set of integers between $k$ and $m$, i.e., 
\[
	\llbracket k,m \rrbracket \coloneqq \left\{ s \in \nn \mid k \leq s \leq m \right\}.
\]
Given $L, L_1, \ldots, L_n \subseteq \mathbb{Z}$, we denote by $L_1 + \cdots + L_n$ the set $\left\{ l_1 + \cdots + l_n \mid l_i \in L_i \right\}$ and, for $l \in \zz$, we set $l + L \coloneqq \{l\} + L$.

\subsection{Exponential Puiseux Semirings}

Unless we specify otherwise, a \emph{monoid} is defined to be a semigroup with identity that is cancellative, commutative, and \emph{reduced} (i.e., its only invertible element is the identity) and, throughout this paper, we use additive notation for monoids. Now let $M$ be a monoid. We denote by $\mathcal{A}(M)$ the set consisting of elements $a \in M^{\bullet} \coloneqq M\setminus\{0\}$ satisfying that if $a = x + y$ for some $x,y \in M$ then either $x = 0$ or $y = 0$; the elements of this set are called \emph{atoms}. For a subset $G \subseteq M$, we denote by $\langle G \rangle$ the smallest submonoid of $M$ containing $G$, and if $M = \langle G \rangle$ then it is said that $G$ is a \emph{generating set} of $M$. A monoid $M$ is \emph{atomic} if $M = \langle\mathcal{A}(M)\rangle$. For $x,y \in M$, it is said that $x$ \emph{divides} $y$ if there exists $x' \in M$ such that $y = x + x'$\! in which case we write $x \,|_M \,y$ and drop the subscript whenever $M = (\nn, \times)$. A subset $I$ of $M$ is an \emph{ideal} of $M$ provided that $I + M \subseteq I$. An ideal $I$ is \emph{principal} if $I = x + M$ for some $x \in M$. Furthermore, it is said that $M$ satisfies the \emph{ascending chain condition on principal ideals} (or \emph{ACCP}) if every increasing sequence of principal ideals of $M$ eventually stabilizes. 

A \emph{numerical monoid} $S$ is an additive submonoid of $\nn$ whose complement in $\nn$ is finite; the greatest integer that is not an element of $S$ is called the \emph{Frobenius number} of $S$ and is denoted by $F(S)$. It is well known that numerical monoids are always finitely generated and, therefore, atomic. An introduction to numerical monoids can be found in~\cite{GSJCR2009}. On the other hand, \emph{Puiseux monoids} are additive submonoids of $\qq_{\ge 0}$, so they are natural generalizations of numerical monoids. The factorization invariants of these monoids have received considerable attention during the past five years (see, for instance,~\cite{fG2019, FGCO, MGotti}). A particularly captivating class of Puiseux monoids is the one consisting of monoids generated by the nonnegative powers of a positive rational number. These monoids, called \emph{rational cyclic semirings}, were first considered in \cite{GG17}, and a detailed description of their factorization invariants was provided in \cite{CGG19}. See \cite{CGG2019} for a friendly survey about Puiseux monoids.

\begin{definition} \label{def: exponential PM}
	Take $r \in \qq_{>0}$ and let $S$ be a numerical monoid. We let $M_{r,S}$ denote the monoid
	\[
	M_{r,S} \coloneqq \left\langle r^k \mid k \in S \right\rangle\!,
	\]
	which we call \emph{exponential Puiseux semiring}. For $n \in \nn$, we denote by $s_n$ the $(n+1)$th smallest element of $S$ and set $\delta_n \coloneqq s_{n+1} - s_n$.
\end{definition}

With notation as in Definition~\ref{def: exponential PM}, if $r$ is a positive integer then we say that $M_{r,S}$ (which is just $\nn$) is the \emph{trivial} exponential Puiseux semiring. By virtue of \cite[Proposition~5.1]{SAJBHP2020}, the monoid $M_{r,S}$ is, in fact, a \emph{positive semiring} (i.e., an additive submonoid of the real line containing $1$ and closed under multiplication). Clearly, exponential Puiseux semirings are generalizations of rational cyclic semirings, and most of them are atomic as the next proposition indicates.

\begin{prop} \cite[Proposition~3.7]{SAJBHP2020} \label{prop:atomicity of exponential PMs}
	Let $M_{r,S}$ be a nontrivial exponential Puiseux semiring. The following statements hold. 
	\begin{enumerate}
		\item If $\mathsf{n}(r)>1$ and $\mathsf{d}(r)>1$ then $M_{r,S}$ is atomic and $\mathcal{A}(M_{r,S}) = \{r^s \mid s \in S\}$.
		\item If $\mathsf{n}(r)=1$ and $\mathsf{d}(r)>1$ then $M_{r,S}$ is not atomic and $\mathcal{A}(M_{r,S}) = \emptyset$.
	\end{enumerate}
\end{prop}

\subsection{Factorizations}

For the rest of the section, let $M$ be an atomic monoid. The \emph{factorization monoid} of $M$, denoted by $\mathsf{Z}(M)$, is the free (commutative) monoid on $\mathcal{A}(M)$. The elements of $\mathsf{Z}(M)$ are called \emph{factorizations}, and if $z = a_1 + \cdots + a_n \in \mathsf{Z}(M)$ for $a_1, \ldots, a_n \in\mathcal{A}(M)$ then it is said that the \emph{length} of $z$, denoted by $|z|$, is $n$. We assume that the empty factorization has length $0$. The unique monoid homomorphism $\pi\colon\mathsf{Z}(M) \to M$ satisfying that $\pi(a) = a$ for all $a \in\mathcal{A}(M)$ is called the \emph{factorization homomorphism} of $M$. For each $x \in M$, there are two important sets associated to $x$:
\[
\mathsf{Z}_M(x) \coloneqq \pi^{-1}(x) \subseteq \mathsf{Z}(M) \hspace{0.6 cm}\text{ and } \hspace{0.6 cm}\mathsf{L}_M(x) \coloneqq \{|z| : z \in\mathsf{Z}_M(x)\},
\]
which are called the \emph{set of factorizations} of $x$ and the \emph{set of lengths} of $x$, respectively; we drop the subscript whenever the monoid is clear from the context. Note that $\mathsf{L}(0) = \{0\}$. Additionally, the \emph{system of sets of lengths} of $M$ is defined by
\[
	\mathcal{L}(M) \coloneqq \{\mathsf{L}(x) \mid x \in M\}.
\]
See~\cite{aG16} for a survey on sets of lengths. It is said that $M$ is a \emph{finite factorization monoid} (or an FFM) if $\mathsf{Z}(x)$ is nonempty and finite for all $x \in M$. Similarly, $M$ is a \emph{bounded factorization monoid} (or BFM) if $\mathsf{L}(x)$ is nonempty and finite for all $x \in M$. We say that an element $x \in M$ \emph{divides} a factorization $z \in \mathsf{Z}(M)$ provided that $x$ divides $\pi(z)$ in $M$.

\subsection{Presentations and Betti Elements}

Let $\sigma$ be a subset of $M \times M$. Then we set 
\[
	\sigma^{t} := \{(s + x, r + x) \mid (s,r) \in \sigma \ \text{and} \ x \in M\}.
\]
On the other hand, we let $\sigma^{e}$ denote the smallest (under inclusion) equivalence relation on $M$ containing $\sigma$. The relation $\sigma$ on $M$ is called a \emph{congruence} provided that $\sigma$ is an equivalence relation satisfying that $\sigma^t = \sigma$. The congruence \emph{generated by} $\sigma$, denoted by $\sigma^\sharp$, is the smallest (under inclusion) congruence on $M$ containing $\sigma$. 

\begin{prop} \cite[Proposition~1.5.9]{howie} \label{prop:congruence generated description}
	If $M$ is a monoid and $\rho \subseteq M \times M$ then $\rho^\sharp = (\rho^t)^e$\!.
\end{prop}

Let $\psi \colon M \to M'$ be a monoid homomorphism. Then $\ker \psi := \{(a,b) \in M \times M \mid \psi(a) = \psi(b)\}$ is called the \emph{kernel congruence} of $M$ with respect to $\psi$. The kernel congruence of the factorization homomorphism $\pi \colon \mathsf{Z}(M) \to M$ is denoted by $\sim_M$. If $\rho \subseteq \mathsf{Z}(M) \times \mathsf{Z}(M)$ generates $\sim_M$ in the sense that $\sim_M$ is the smallest (under inclusion) congruence of $\mathsf{Z}(M)$ containing $\rho$ then it is said that $\rho$ is a \emph{presentation} of $M$. 

Given two factorizations $z = \sum_{a \in \mathcal{A}(M)} \alpha_a a$ and $z' = \sum_{a \in \mathcal{A}(M)} \beta_a a$ in $\mathsf{Z}(x)$, we set $\gcd(z,z') \coloneqq \sum_{a \in \mathcal{A}(M)} \min(\alpha_a, \beta_a)a$ and $\mathsf{d}(z,z') \coloneqq \max \{|z|, |z'| \}- |\gcd(z,z')|$. For $x \in M$ and $z,z' \in \mathsf{Z}(x)$, a \emph{connecting chain of factorizations} or, simply, a  \emph{connecting chain} from $z$ to $z'$ is an ordered chain of factorizations $z = z_1,\dots,z_n = z' \in \mathsf{Z}(x)$ such that $\gcd(z_i, z_{i+1}) \in \mathsf{Z}(M)^\bullet$ for each $i \in \llbracket 1, n - 1 \rrbracket$. Now consider the binary relation $\mathcal{R}$ on $\mathsf{Z}(M)$ defined as follows: $(z,z') \in \mathcal{R}$ if there exists a connecting chain from $z$ to $z'$\!. Evidently, $\mathcal{R} \subseteq \ker \pi$. Given $x \in M$, we denote by $\mathcal{R}_x$ the set consisting of all $\mathcal{R}$-classes of $\mathsf{Z}(x)$, and $x$ is called a \emph{Betti element} of $M$ provided that $|\mathcal{R}_x| > 1$. The set of Betti elements of $M$ is denoted by $\betti(M)$.

\section{Betti Elements, Catenary Degrees, and Sets of Distances}

Several papers have been devoted to study the \emph{catenary degree} (see definition below) in atomic monoids (e.g., \cite{CGSLL2012,CGSLLPR2006,AGDGWAS2011,AGQZ2015}). However, exact formulas to calculate this invariant are hard to come by, even in the context of finitely generated monoids.

\begin{definition}
	Let $M$ be an atomic monoid, and let $x$ be an element of $M$.
	
	\begin{enumerate}
		\item For $n \in \overline{\nn}$, a finite sequence $z_1, \ldots, z_k \in \mathsf{Z}(x)$ is called an \emph{$n$-chain of factorizations} connecting $z_1$ with $z_k$ if $\mathsf{d}(z_{i},z_{i + 1}) \leq n$ for each $i \in \llbracket 1,k - 1 \rrbracket$. 
		\item The \emph{catenary degree} of $x$, denoted by $\mathsf{c}(x)$, is the smallest $n \in \overline{\nn}$ for which any two factorizations in $\mathsf{Z}(x)$ are connected through an $n$-chain. 
		\item The \emph{catenary degree} of $M$\!, denoted by $\mathsf{c}(M)$, is defined by $\mathsf{c}(M) \coloneqq \sup\{\mathsf{c}(x) \mid x \in M\}$.
	\end{enumerate}
\end{definition}

In this section, we describe the $\mathcal{R}$-classes of nontrivial atomic exponential Puiseux semirings. This will allow us to compute their catenary degrees via \cite[Corollary~9]{philipp2010}, an approach considerable different than that adopted in \cite[Corollary~3.4]{CGG19} to compute the catenary degrees of atomic rational cyclic semirings. But first let us collect two technical lemmas; the first one was borrowed from \cite{SAJBHP2020} (Lemma~3.8), and the second one is its counterpart for the case where $r > 1$.

\begin{lemma} \cite[Lemma~3.8]{SAJBHP2020} \label{lemma: algebraic manipulations for r < 1}
	Let $x$ be a nonzero element of an atomic exponential Puiseux semiring $M_{r,S}$ with $r \in \qq_{<1}$, and consider a factorization $z=\sum_{i=0} ^n c_i r^{s_i} \in\mathsf{Z}(x)$ with coefficients $c_0, \ldots, c_n \in \nn$. The following conditions hold.
	\begin{enumerate}
		\item $\min \mathsf{L}(x) =  |z|$ if and only if $c_i < \mathsf{d}(r)^{\delta_{i-1}}$ for each $i \in \llbracket 1,n \rrbracket$. 
		\item There exists exactly one factorization $z_{\min} \in \mathsf{Z}(x)$ of minimum length.
		\item $\max \mathsf{L}(x) = |z|$ if and only if $c_i < \mathsf{n}(r)^{\delta_i}$ for each $i\in\llbracket 0,n \rrbracket$.
		\item There exists, at most, one factorization $z_{\max} \in \mathsf{Z}(x)$ of maximum length.
		\item If $c_i < \mathsf{n}(r)$ for every $i \in \llbracket 0,n \rrbracket$ then $| \mathsf{Z}(x)| =1$.
	\end{enumerate}
\end{lemma}

\begin{lemma} \label{lemma: algebraic manipulations for r > 1}
	Let $x$ be a nonzero element of an atomic exponential Puiseux semiring $M_{r,S}$ with $r \in \qq_{>1} \setminus \nn$, and consider a factorization $z=\sum_{i=0} ^n c_i r^{s_i} \in\mathsf{Z}(x)$ with coefficients $c_0, \ldots, c_n \in \nn$. The following conditions hold.
	\begin{enumerate}
		\item $\min \mathsf{L}(x) =  |z|$ if and only if $c_i < \mathsf{n}(r)^{\delta_{i}}$ for each $i \in \llbracket 0,n \rrbracket$. 
		\item There exists exactly one factorization $z_{\min} \in \mathsf{Z}(x)$ of minimum length.
		\item $\max \mathsf{L}(x) = |z|$ if and only if $c_i < \mathsf{d}(r)^{\delta_{i - 1}}$ for each $i\in\llbracket 1,n \rrbracket$.
		\item There exists exactly one factorization $z_{\max} \in \mathsf{Z}(x)$ of maximum length.
		\item If $c_i < \mathsf{d}(r)$ for every $i \in \llbracket 0,n \rrbracket$ then $| \mathsf{Z}(x)| =1$.
	\end{enumerate}
\end{lemma}

\begin{proof}
	The proofs of $(1)$, $(2)$, $(3)$, and $(4)$ are left to the reader as they mimick the proof of the corresponding parts of~\cite[Lemma~3.2]{CGG19}. And $(5)$ readily follows from $(1)$ and $(3)$.
\end{proof}

For the rest of the paper we assume, without explicitly mentioning, that given an element $x$ in an atomic exponential Puiseux semiring $M_{r,S}$, the set $\mathsf{Z}(x)$ contains exactly one factorization of minimum length; we also assume that $\mathsf{Z}(x)$ contains exactly one factorization of maximum length provided that $r > 1$.

\begin{remark} \label{remark: any factorization can be transformed into the factorization of minimum length}
	Let $M_{r,S}$ be a nontrivial atomic exponential Puiseux semiring, and let $x$ be an element of $M_{r,S}$. Given $z \in \mathsf{Z}(x)$, there exist factorizations $z = z_1, \ldots, z_k = z_{\min} \in \mathsf{Z}(x)$ such that $z_{\min}$ is the factorization of minimum length of $x$ and, for each $i \in \llbracket 1, k - 1 \rrbracket$, we have $|z_{i}| - |z_{i + 1}| = |\mathsf{n}(r)^{\delta_m} - \mathsf{d}(r)^{\delta_m}|$ for some $m \in \nn$. We can generate such a sequence of factorizations iteratively: set $z_1 \coloneqq z$, and assume that we already defined $z_m  = \sum_{i = 0}^{n} c_i r^{s_i} \in \mathsf{Z}(x)$ for some $m \in \nn^{\bullet}$. If $z_m$ is not the factorization of minimum length of $x$ and $r < 1$ (resp. $r > 1$) then $c_j \geq \mathsf{d}(r)^{\delta_{j - 1}}$ (resp. $c_j \geq \mathsf{n}(r)^{\delta_j}$) for some $j \in \llbracket 1,n \rrbracket$ (resp. $j \in \llbracket 0,n \rrbracket$) by Lemma~\ref{lemma: algebraic manipulations for r < 1} (resp. Lemma~\ref{lemma: algebraic manipulations for r > 1}). By applying the transformation
	\[
	\mathsf{d}(r)^{\delta_{j - 1}}r^{s_{j}} = \mathsf{n}(r)^{\delta_{j - 1}}r^{s_{j - 1}}\, (\text{resp. } \mathsf{n}(r)^{\delta_{j}}r^{s_{j}} = \mathsf{d}(r)^{\delta_{j}}r^{s_{j + 1}}),
	\]
	we can generate a factorization $z_{m + 1} \in \mathsf{Z}(x)$ satisfying $|z_m| - |z_{m + 1}| = |\mathsf{n}(r)^{\delta_t} - \mathsf{d}(r)^{\delta_t}|$ for some $t \in \nn$. Consequently, the inequality $|z_m| > |z_{m + 1}|$ holds which, in turn, implies that these iterations eventually stop. By Lemma~\ref{lemma: algebraic manipulations for r < 1} (resp. Lemma~\ref{lemma: algebraic manipulations for r > 1}), the last factorization we obtained is precisely the factorization of minimum length of $x$. A similar statement is true, \emph{mutatis mutandis}, if we take $r > 1$ and $z_k \in \mathsf{Z}(x)$ to be the factorization of maximum length of $x$. 
\end{remark}  

The following example illustrates how to go from one factorization to another of the same element using the transformations described in Remark~\ref{remark: any factorization can be transformed into the factorization of minimum length}.

\begin{example}
	Let $r = 2/3$ and $S = \{0, 5, 6, 8, 9\} \cup \nn_{\geq 11}$, and consider the Puiseux semiring $M_{r,S}$. Clearly, $z = 3(2/3)^6 + 2(2/3)^8$, $z_{\min} = 2(2/3)^5 + 2(2/3)^8$, and $z' = 2(2/3)^5 + 3(2/3)^9$ are factorizations of the same element $x \in M_{r,S}$. Note that $z_{\min}$ is, in fact, the factorization of minimum length of $x$. By performing the transformations $3(2/3)^6 = 2(2/3)^5$ and $2(2/3)^8 = 3(2/3)^9$, we can go from the factorization $z$ to $z_{\min}$ and from $z_{\min}$ to $z'$, respectively. 
\end{example}

We are now in a position to describe the $\mathcal{R}$-classes of nontrivial atomic exponential Puiseux semirings.

\begin{prop} \label{prop: betti elements}
	Let $M_{r,S}$ be a nontrivial atomic exponential Puiseux semiring. Let $x \in M_{r,S}$, and set $R_x(r^{s_m}) \coloneqq \mathsf{Z}(x) \cap (r^{s_m} + \mathsf{Z}(M_{r,S}))$ for $m \in \nn$. The following statements hold.
	\begin{enumerate}
		\item $\betti\,(M_{r,S}) = \left\{\mathsf{n}(r)^{\delta_n}r^{s_n} \mid n \in \nn\right\}$.
		\item If $x = \mathsf{n}(r)^{\delta_n}r^{s_n}$ for some $n \in \nn$ then $\mathcal{R}_{x} = \left\{R_x(r^{s_n}), R_x(r^{s_{n + 1}})\right\}$.
	\end{enumerate}
\end{prop}

\begin{proof}
	Fix $n \in \nn$, and take $x = \mathsf{n}(r)^{\delta_n}r^{s_n} \in M_{r,S}$. Let $A = \{r^{s_0}\!, \ldots, r^{s_n}\}$ and $B = \mathcal{A}(M_{r,S})\setminus A$. Consider the factorizations $z = \mathsf{n}(r)^{\delta_n}r^{s_n}$ and $z' = \mathsf{d}(r)^{\delta_n}r^{s_{n + 1}}$ in $\mathsf{Z}(x)$. By Remark~\ref{remark: any factorization can be transformed into the factorization of minimum length}, if $r < 1$ (resp. $r > 1$) then the factorization of minimum (resp. maximum) length of $x$ does not contain atoms from the subset $B$. By way of contradiction, suppose that there exists a factorization $z'' \in \mathsf{Z}(x)$ containing atoms from both subsets $A$ and $B$. If $r < 1$ (resp. $r > 1$) then by repeatedly applying the identity $\mathsf{d}(r)^{\delta_m}r^{s_{m + 1}} = \mathsf{n}(r)^{\delta_m} r^{s_m}$ we can generate factorizations $z'' = z_1, \ldots, z_k \in \mathsf{Z}(x)$ such that $z_k$ is the factorization of minimum (resp. maximum) length of $x$ and, for each $i \in \llbracket 1, k - 1 \rrbracket$, we have $|z_{i}| - |z_{i + 1}| = |\mathsf{n}(r)^{\delta_m} - \mathsf{d}(r)^{\delta_m}|$ (resp. $|z_{i + 1}| - |z_{i}| = |\mathsf{n}(r)^{\delta_m} - \mathsf{d}(r)^{\delta_m}|$) for some $m \in \nn$. But since $z''$ contains atoms from $B$ and $z_k$ does not, at some point in generating $z_1, \ldots, z_k$ we applied the transformation $\mathsf{d}(r)^{\delta_n}r^{s_{n + 1}} = \mathsf{n}(r)^{\delta_n} r^{s_n}$\!, which is a contradiction given that $z''$ also contains atoms from $A$. Hence $z$ and $z'$ belong to different $\mathcal{R}$-classes of $\mathsf{Z}(x)$, from which the inclusion $\{\mathsf{n}(r)^{\delta_n}r^{s_n} \mid n \in \nn\} \subseteq\betti\,(M_{r,S})$ follows. 
	
	Note that if the atoms ocurring in a factorization $z^* \in \mathsf{Z}(x)$ are in $A$ then, by a similar argument to that one used in Remark~\ref{remark: any factorization can be transformed into the factorization of minimum length}, there exist factorizations $z = z^*_1, \ldots, z^*_t = z^* \in \mathsf{Z}(x)$ such that, for each $i \in \llbracket 2,t \rrbracket$, one can obtain $z^*_i$ from $z^*_{i - 1}$ by performing a transformation of the form $\mathsf{d}(r)^{\delta_m}r^{s_{m + 1}} = \mathsf{n}(r)^{\delta_m} r^{s_m}$\! for some $m \in \nn$. Since $\mathsf{d}(r)^{\delta_{k - 1}} \nmid \mathsf{n}(r)^{\delta_k}$ for any $k \in \nn^{\bullet}$\!, the atom $r^{s_n}$ shows up in the factorization $z^*_i$ for each $i \in \llbracket 1,t \rrbracket$. Consequently, all factorizations of $x$ with atoms in $A$ are in the same $\mathcal{R}$-class, namely $R_x(r^{s_n})$. Similarly, all factorizations of $x$ with atoms in $B$ are in $R_x(r^{s_{n + 1}})$, from which $(2)$ follows.
	
	To verify that the inclusion $\betti\,(M_{r,S}) \subseteq \{\mathsf{n}(r)^{\delta_n}r^{s_n} \mid n \in \nn\}$ holds, let $x \in M_{r,S} \setminus \{\mathsf{n}(r)^{\delta_n} r^{s_n} \mid n \in \nn\}$, and let $z_{\min} \in \mathsf{Z}(x)$ be the factorization of minimum length of $x$. For every $z \in \mathsf{Z}(x)$, there exist factorizations $z = z_1, \ldots, z_k = z_{\min} \in \mathsf{Z}(x)$ such that, for each $i \in \llbracket 1, k - 1 \rrbracket$, we have $|z_{i}| - |z_{i + 1}| = |\mathsf{n}(r)^{\delta_m} - \mathsf{d}(r)^{\delta_m}|$ for some $m \in \nn$ by Remark~\ref{remark: any factorization can be transformed into the factorization of minimum length}; since we generate such a sequence of factorizations through transformations of the form $\mathsf{n}(r)^{\delta_k} r^{s_k} = \mathsf{d}(r)^{\delta_k}r^{s_{k + 1}}$\!, we have $\gcd(z_i,z_{i + 1}) \neq 0$ unless $z_i = \mathsf{n}(r)^{\delta_{j}}r^{s_j}$ or $z_{i + 1} = \mathsf{n}(r)^{\delta_j}r^{s_j}$ for some $j \in \nn$, but this is impossible since $x \not\in \{\mathsf{n}(r)^{\delta_n} r^{s_n} \mid n \in \nn\}$. Consequently, there is a chain connecting $z$ and $z_{\min}$ which, in turn, implies $|\mathcal{R}_x| = 1$, and our proof concludes.
\end{proof}

\begin{cor}
	Let $M_{r,S}$ be a nontrivial atomic exponential Puiseux semiring. Then
	\[
	\mathsf{c}(M_{r,S}) = 
	\max\left(\mathsf{n}(r), \mathsf{d}(r)\right)^{\delta_0}\!. 
	\]
\end{cor}

\begin{cor} \label{cor: catenary degree of rational cyclic semirings}
	Let $M_{r,\nn}$ be a nontrivial atomic rational cyclic semiring. Then
	\[
	\mathsf{c}(M_{r,\nn}) = 
	\max(\mathsf{n}(r), \mathsf{d}(r)).
	\]
\end{cor}

For the rest of the section, we focus on the factorization invariant known as \emph{set of distances} or \emph{delta set}. This invariant has been extensively studied in the context of BFMs, including numerical monoids \cite{CBSTCNKDR2006} and transfer Krull monoids \cite{AGQZ2019}. 

\begin{definition} \label{def: set of distances}
	Let $M$ be an atomic monoid, and let $x$ be an element of $M$. 
	\begin{enumerate}
		\item It is said that $d \in \nn^{\bullet}$ is a \emph{distance} of $x$ provided that $\mathsf{L}(x) \cap \llbracket l, l + d \rrbracket = \{l, l + d\}$ for some $l \in \mathsf{L}(x)$. 
		\item The \emph{set of distances of $x$}, denoted by $\Delta(x)$, is the set consisting of all the distances of $x$.
		\item The set	
		$\Delta(M) \coloneqq \bigcup_{x \in M} \Delta(x)$ is called the \emph{set of distances} of $M$\!.
	\end{enumerate}
\end{definition}

We now show that atomic exponential Puiseux semirings have finite sets of distances. Additionally, we compute the minimum and maximum of $\Delta(M_{r,S})$ and proved that they are both attained at Betti elements for all atomic exponential Puiseux semirings $M_{r,S}$.

\begin{prop} \label{prop: delta sets of exponential Puiseux semirings}
	Let $M_{r,S}$ be a nontrivial atomic exponential Puiseux semiring. Then 
	\begin{equation} \label{eq: inclusions of delta sets}
	\big\{|\mathsf{n}(r)^{\delta_n} - \mathsf{d}(r)^{\delta_n}| : n \in \nn\big\} \subseteq \Delta(M_{r,S}) \subseteq \big\llbracket |\mathsf{n}(r) - \mathsf{d}(r)|, |\mathsf{n}(r)^{\delta_0} - \mathsf{d}(r)^{\delta_0}| \big\rrbracket.
	\end{equation}
\end{prop}

\begin{proof}
	The first inclusion and the fact that $\Delta(M_{r,S})$ is finite follow almost immediately from Remark~\ref{remark: any factorization can be transformed into the factorization of minimum length}. On the other hand, it is easy to mimic the proof of \cite[Theorem~2.5]{CGSLLMS2012} to show that $\max\Delta(M_{r,S})$ is achieved at a Betti element. We leave all these details to the reader. Our result follows then from Remark~\ref{remark: any factorization can be transformed into the factorization of minimum length} and Proposition~\ref{prop: betti elements}.
\end{proof}

\begin{cor} \label{cor: delta sets of rational cyclic semirings}
	Let $M_{r,\nn}$ be a nontrivial atomic rational cyclic semiring. Then $\Delta(M_{r,\nn}) = \{|\mathsf{n}(r) - \mathsf{d}(r)|\}$.
\end{cor}

The inclusions in \eqref{eq: inclusions of delta sets} might be proper. Consider the following example.

\begin{example}
	Let $r = 2/3$, and consider the following numerical monoids 
	\[
		S = \{0, 18, 19, 25, 27\} \cup \nn_{\geq 36} \hspace{.6cm} \text{ and } \hspace{.6cm} S' = \nn_{\geq 2} \cup\{0\}.
	\]
	Clearly, the monoids $M_{r,S}$ and $M_{r,S'}$ are both atomic. Moreover, it is not hard to check 
	\[
	\mathsf{L}_{M_{r,S}}\left(2(2/3)^{18} + 4(2/3)^{25}\right) = \left\{6, 7, 11, 12\right\}\!,
	\]
	which implies that $4 \in \Delta(M_{r,S})$, but $4 \not\in \left\{|\mathsf{n}(r)^{\delta_n} - \mathsf{d}(r)^{\delta_n}| : n \in \nn\right\}$, where the $\delta_n\!$'s are taken with respect to the numerical monoid $S$. On the other hand, we have $\Delta(M_{r,S'}) = \{1,5\}$ as the reader can verify, so $ \Delta(M_{r,S'}) \subsetneq \big\llbracket |\mathsf{n}(r) - \mathsf{d}(r)|, |\mathsf{n}(r)^{\delta_0} - \mathsf{d}(r)^{\delta_0}| \big\rrbracket$, where $\delta_0 = 2$.
\end{example}

\begin{prop}
	Let $M_{r,S}$ be a nontrivial atomic exponential Puiseux semiring. Then
	\[
	\min\Delta(M_{r,S}) = \left|\mathsf{n}(r) - \mathsf{d}(r)\right| \hspace{0.6 cm} \text{ and } \hspace{0.6 cm} \max\Delta(M_{r,S}) = \left|\mathsf{n}(r)^{\delta_0} - \mathsf{d}(r)^{\delta_0}\right|
	\]
	are both attained at Betti elements.
\end{prop}

\begin{proof}
	By Proposition~\ref{prop: delta sets of exponential Puiseux semirings}, we have that the equalities $\min\Delta(M_{r,S}) = |\mathsf{n}(r) - \mathsf{d}(r)|$ and $\max\Delta(M_{r,S}) = |\mathsf{n}(r)^{\delta_0} - \mathsf{d}(r)^{\delta_0}|$ hold. Furthermore, we have that $|\mathsf{n}(r) - \mathsf{d}(r)|$ (resp. $|\mathsf{n}(r)^{\delta_0} - \mathsf{d}(r)^{\delta_0}|$) is a distance of $\mathsf{n}(r)r^{F(S) + 1}$ (resp. $\mathsf{n}(r)^{\delta_0}r^{s_0}$). Our argument concludes after noticing that both elements $\mathsf{n}(r)r^{F(S) + 1}$ and $\mathsf{n}(r)^{\delta_0}r^{s_0}$ are Betti elements of $M_{r,S}$ by Proposition~\ref{prop: betti elements}.
\end{proof}

In \cite{CGSLLMS2012}, it was proved that, for every bounded factorization monoid $M$, the minimum and maximum of $\Delta(M)$ are both attained at Betti elements. But notice that atomic exponential Puiseux semirings are not, in general, BFMs (see \cite[Example~4.7]{SAJBHP2020}).

\section{Sets of Lengths and Their Unions}

For a positive integer $d$ and a nonnegative integer $B$, a subset $L \subseteq \mathbb{Z}$ is called an \emph{almost arithmetic progression} (or \emph{AAP}) with \emph{difference} $d$ and \emph{bound} $B$ if
	\[
	L = y + (L' \cup L^{*} \cup L'') \subseteq y + d\mathbb{Z},
	\] 
where $y \in\mathbb{Z}$ and $L^{*}$ is a nonempty arithmetic progression with difference $d$ such that $\min L^* = 0$, $L' \subseteq [-B,-1]$, and $L'' \subseteq \sup L^* + [1,B]$. If the set $L^*$ has infinite cardinality then we assume that $L'' = \emptyset$. As the name indicates, almost arithmetic progressions are generalizations of arithmetic sequences, and they have been used in factorization theory to describe sets of lengths of various classes of monoids (see, for instance, \cite[Theorem~4.3.6]{GH06b} and \cite[Theorem~4.2]{GaGe09}).

In this section, we show that sets of lengths of atomic exponential Puiseux semirings are well structured. Specifically, we prove that if $M_{r,S}$ is a nontrivial atomic exponential Puiseux semiring then there exists $B \in \nn$ such that every $L \in \mathcal{L}(M_{r,S})$ is an AAP with difference $|\mathsf{n}(r) - \mathsf{d}(r)|$ and bound $B$. 

\begin{theorem} \label{theorem: sets of lengths}
	Let $M_{r,S}$ be a nontrivial atomic exponential Puiseux semiring. There exists $B \in \nn$ such that every $L \in \mathcal{L}(M_{r,S})$ is an AAP with difference $|\mathsf{n}(r) - \mathsf{d}(r)|$ and bound~$B$\!.
\end{theorem}

\begin{proof}
	Clearly, there exists $m \in \nn$ such that $s_m = F(S) + 1$, where $F(S)$ represents the Frobenius number of $S$. Let $x$ be an arbitrary element of $M_{r,S}$, and let $z_{\min} = \sum_{i = 0}^{n} c_i r^{s_i} \in \mathsf{Z}(x)$ with coefficients $c_0, \ldots, c_n \in \nn$ be the factorization of minimum length of $x$. Obviously, we may assume that $m \leq n$. Before continuing with the proof, we introduce a definition. 
	Given a factorization $z = \sum_{i = 0}^{k} d_i r^{s_i}$ with coefficients $d_0, \ldots, d_k \in \nn$, we say that the sub-factorizations $\sum_{i = 0}^{m} d_i r^{s_i}$ and $\sum_{i = m + 1}^{k} d_i r^{s_i}$ are the \emph{prefix} and \emph{suffix} of $z$, respectively. We distinguish two cases. 
	
	\vskip 0.1cm
	\textsc{Case 1:} $r < 1$. Let $B_1 = \mathsf{d}(r)^{s_m} - \mathsf{n}(r)^{s_m}$\!, and let us denote by $B_2$ the maximum number of times we can consecutively apply the identity $\mathsf{n}(r)^{\delta_\alpha}r^{s_\alpha} = \mathsf{d}(r)^{\delta_\alpha}r^{s_{\alpha + 1}}$ with $\alpha \in \llbracket 0, m - 1 \rrbracket$ to increase the length of the prefix of a factorization $z = \sum_{i = 0}^{k} d_i r^{s_i}$ satisfying that $d_i < \mathsf{n}(r)^{s_m - s_i}$ for each $i \in \llbracket 0,m - 1 \rrbracket$. Evidently, neither $B_1$ nor $B_2$ depends on any element of $M_{r,S}$. Set 
	\[
		B \coloneqq \max\left(B_1, B_2\cdot\left(\mathsf{d}(r)^{\delta_0} - \mathsf{n}(r)^{\delta_0}\right)\right).
	\]
	We argue that $\mathsf{L}(x)$ is an AAP with difference $\mathsf{d}(r) - \mathsf{n}(r)$ and bound $B$. Observe that if $S = \nn$ then $B = 0$. Now if $c_j \geq \mathsf{n}(r)^{s_m - s_j}$ for some $j \in \llbracket 0,m - 1 \rrbracket$  then, by carrying out the transformation $\mathsf{n}(r)^{s_m - s_j}r^{s_j} = \mathsf{d}(r)^{s_m - s_j}r^{s_m}$\!, we can generate a new factorization $z' \in \mathsf{Z}(x)$ such that $|z'| - |z| \leq B$. Also note that, in this case, the term $\mathsf{d}(r)^{s_m - s_j}r^{s_m}$ shows up as a summand in the factorization $z'$; consequently, we can generate new factorizations of $x$ by applying, any number of times, the identity $\mathsf{n}(r)r^{\alpha} = \mathsf{d}(r)r^{\alpha + 1}$\! with $\alpha \geq s_m$. Hence
	\[
		\left\{|z'| + l\!\cdot\!(\mathsf{d}(r) - \mathsf{n}(r)) \mid l \in \nn\right\} \subseteq \mathsf{L}(x),
	\]
	and the set $\mathsf{L}(x)$ is an AAP with difference $\mathsf{d}(r) - \mathsf{n}(r)$ and bound $B$. On the other hand, if $c_j < \mathsf{n}(r)^{s_m - s_j}$ for each $j \in \llbracket 0,m - 1 \rrbracket$ then consider any factorization $z' \in \mathsf{Z}(x)$ satisfying that $|z'| - |z_{\min}| > B$. By Remark~\ref{remark: any factorization can be transformed into the factorization of minimum length}, there exist factorizations $z_{\min} = z_1, \ldots, z_k = z' \in \mathsf{Z}(x)$ such that, for each $i \in \llbracket 1, k - 1 \rrbracket$, we have 
	\[
		|z_{i + 1}| - |z_{i}| = \mathsf{d}(r)^{\delta_n} - \mathsf{n}(r)^{\delta_n}
	\]
	for some $n \in \nn$. Since $|z'| - |z_{\min}| > B_2\cdot\left(\mathsf{d}(r)^{\delta_0} - \mathsf{n}(r)^{\delta_0}\right)$, there exists $j \in \llbracket 1, k - 1 \rrbracket$ such that $z_{j + 1}$ was obtained from $z_j$ by carrying out a transformation of the form $\mathsf{n}(r)^{\delta_t}r^{s_t} = \mathsf{d}(r)^{\delta_t}r^{s_{t + 1}}$ for some $t \geq m$. As before, a term $\mathsf{n}(r)^{\delta_t}r^{s_{t}}$ shows up in the factorization $z_{j}$, which means that we can generate new factorizations of $x$ by applying the identity $\mathsf{n}(r)r^{\alpha} = \mathsf{d}(r)r^{\alpha + 1}$ with $\alpha \geq s_m$ an arbitrary number of times. Since $\mathsf{d}(r) - \mathsf{n}(r) \,|\, \mathsf{d}(r)^{\delta_n} - \mathsf{n}(r)^{\delta_n}$ for every $n \in \nn$, we have that
	\[
		|z'| \in \left\{|z_{j}| + l\cdot \left(\mathsf{d}(r) - \mathsf{n}(r)\right) \mid l \in \nn\right\} \subseteq \mathsf{L}(x).
	\]
	Observe that $|z_j| < |z'|$. If $|z_j| - |z_{\min}| > B$ then, by the same token, there exists $z_t \in \mathsf{Z}(x)$ such that 
	\[
		|z_j| \in \{|z_{t}| + l\cdot \left(\mathsf{d}(r) - \mathsf{n}(r)\right) \mid l \in \nn\} \subseteq \mathsf{L}(x)
	\]
	and $|z_t| < |z_j|$. Our argument follows inductively. 

	\vskip 0.1cm
	\textsc{Case 2:} $r > 1$. Let us denote by $B_2$ the maximum number of times we can consecutively apply the identity $\mathsf{d}(r)^{\delta_{\alpha - 1}}r^{s_{\alpha}} = \mathsf{n}(r)^{\delta_{\alpha - 1}}r^{s_{\alpha - 1}}$ with $\alpha \in \llbracket 1, m \rrbracket$ to increase the length of the prefix of a factorization $z = \sum_{i = 0}^{k} d_i r^{s_i}$ satisfying that $d_i < \mathsf{n}(r)^{\delta_i}$ for each $i \in \llbracket 1,m \rrbracket$. The number $B_2$ does not depend on any element $x \in M_{r,S}$. Notice that if $S = \nn$ then $B_2 = 0$. Consider the factorization $z_1 = z' + \sum_{i = m + 1}^{n} c_i r^{s_i} \in \mathsf{Z}(x)$, where $z'$ is the factorization of maximum length of $\pi(\sum_{i = 0}^{m} c_i r^{s_i})$. It is not hard to see that the inequality $|z_1| - |z_{\min}| \leq B_2\cdot(\mathsf{n}(r)^{\delta_0} - \mathsf{d}(r)^{\delta_0})$ holds. By \cite[Theorem~5.6]{fG16}, the monoid $M_{r,S}$ is an FFM, so there is no loss in assuming that no atom bigger than $r^{s_n}$ divides $x$ in $M_{r,S}$.

Next we describe a procedure to generate, iteratively, overlapping arithmetic sequences of lengths of $x$ with difference $\mathsf{n}(r) - \mathsf{d}(r)$. Set $z_{1}^* \coloneqq z_1$, and let us denote by $K_1$ the maximum number of times we can consecutively apply the identity $\mathsf{d}(r)^{\delta_{\alpha - 1}}r^{s_{\alpha}} = \mathsf{n}(r)^{\delta_{\alpha - 1}}r^{s_{\alpha - 1}}$ with $\alpha \in \llbracket m + 1,n \rrbracket$ to increase the length of $z_{1}^*$. Additionally, set 
\[
	\sigma_1 \coloneqq \left\{|z_{1}^*| + l \cdot (\mathsf{n}(r) - \mathsf{d}(r)) \mid l \in \llbracket 0, K_1 \rrbracket\right\},
\]
and note that $\sigma_1 \subseteq \mathsf{L}(x)$. Suppose that, for some $j \in \nn^{\bullet}$, we already defined $\sigma_j, K_j, \text{ and }z_{j}^* = \sum_{i = 0}^{n} c_{j,i}\, r^{s_i} \in \mathsf{Z}(x)$, where $c_{j,i} < \mathsf{d}(r)^{\delta_{i - 1}}$ for each $i \in \llbracket 1,m \rrbracket$, such that
	\begin{equation*} 
			\sigma_j = \left\{|z_{j}^*| + l \cdot (\mathsf{n}(r) - \mathsf{d}(r)) \mid l \in \llbracket 0, K_j \rrbracket\right\} \subseteq \mathsf{L}(x).
	\end{equation*}
	By hypothesis of induction, the nonnegative integer $K_j$ represents the maximum number of times we can consecutively apply the identity $\mathsf{d}(r)^{\delta_{\alpha - 1}}r^{s_{\alpha}} = \mathsf{n}(r)^{\delta_{\alpha - 1}}r^{s_{\alpha - 1}}$ with $\alpha \in \llbracket m + 1,n \rrbracket$ to increase the length of $z_{j}^*$. Now if $z_{j}^*$ is the factorization of maximum length of $x$ (i.e., $K_j = 0$) then our procedure stops at this factorization. Otherwise, there exists $u \in \llbracket m + 1, n \rrbracket$ such that $c_{j,u} \geq \mathsf{d}(r)^{\delta_{u - 1}} = \mathsf{d}(r)$. After consecutively applying the transformations
	\[
		\mathsf{d}(r)r^{s_u} = \mathsf{n}(r)r^{s_{u - 1}}\!, \ldots\,,  \mathsf{d}(r)r^{s_{m + 1}} = \mathsf{n}(r) r^{s_{m}}
	\]
	to increase the length of $z_{j}^*$, we obtain a factorization $z_{j, j + 1} = \sum_{i = 0}^{n} d_{j + 1,i} \,r^{s_i} \in \mathsf{Z}(x)$, where $d_{j + 1, i} < \mathsf{d}(r)^{\delta_{i - 1}}$ for each $i \in \llbracket 1, m - 1 \rrbracket$ and $d_{j + 1, m} < \mathsf{d}(r)^{\delta_{m - 1}} + \mathsf{n}(r)$. Clearly, the inequalities 
	\[
		|z_{j}^*| < |z_{j, j + 1}| \leq |z_{j}^*| + K_j\cdot(\mathsf{n}(r) - \mathsf{d}(r))
	\]
	hold. Consider the factorization $z_{j + 1}^* = z'' + \sum_{i = m + 1}^{n} d_{j + 1, i}\,r^{s_i} \in \mathsf{Z}(x)$, where $z''$ is the factorization of maximum length of $\pi(\sum_{i = 0}^{m} d_{j + 1, i}\,r^{s_i})$. Note that if $S = \nn$ then $z_{j + 1}^* = z_{j, j + 1}$. It is easy to see that $|z_{j, j + 1}| \leq |z_{j + 1}^*|$, and rewriting $z_{j + 1}^*$ as $z_{j + 1}^* = \sum_{i = 0}^{n} c_{j + 1, i}\, r^{s_i}$ we have $c_{j + 1,i} < \mathsf{d}(r)^{\delta_{i - 1}}$ for each $i \in \llbracket 1,m \rrbracket$. Let us denote by $K_{j + 1}$ the maximum number of times we can consecutively apply the identity $\mathsf{d}(r)^{\delta_{\alpha - 1}}r^{s_{\alpha}} = \mathsf{n}(r)^{\delta_{\alpha - 1}}r^{s_{\alpha - 1}}$ with $\alpha \in \llbracket m + 1,n \rrbracket$ to increase the length of $z_{j + 1}^*$. Then set
	\begin{equation*} 
	\sigma_{j + 1} \coloneqq \left\{|z_{j + 1}^*| + l \cdot (\mathsf{n}(r) - \mathsf{d}(r)) \mid l \in \llbracket 0, K_{j + 1} \rrbracket\right\} \subseteq \mathsf{L}(x).
	\end{equation*}
	If $K_{j + 1}\cdot(\mathsf{n}(r) - \mathsf{d}(r)) \geq |z_{j + 1}^*| - |z_{j, j + 1}|$ then we have that
	\[
		|z_{j + 1}^*| \in \sigma_{(j,j + 1)} \coloneqq\{|z_{j, j + 1}| + l\cdot(\mathsf{n}(r) - \mathsf{d}(r)) \mid l \in \llbracket 0, K_{j + 1} \rrbracket\} \subseteq \mathsf{L}(x),
	\]
	where the inclusion holds because the suffixes of $z_{j, j + 1}$ and $z_{j + 1}^*$ coincide. This means that we have three overlapping arithmetic sequences $\sigma_j$, $\sigma_{(j,j + 1)}$, and $\sigma_{j + 1}$ of lengths of $x$ with difference $\mathsf{n}(r) - \mathsf{d}(r)$. Combining these sequences we obtain a new arithmetic sequence $\sigma^*$ starting at $|z_j^*|$, containing $|z_{j + 1}^*|$, and with difference $\mathsf{n}(r) - \mathsf{d}(r)$. Now update $\sigma_j$ as $\sigma^*$\!. On the other hand, if 
	\[
		K_{j + 1}\cdot(\mathsf{n}(r) - \mathsf{d}(r)) < |z_{j + 1}^*| - |z_{j, j + 1}|
	\]
	then our procedure stops at the factorization $z_{j + 1}^*$. Notice that our algorithm eventually stops. In fact, in each step we generate from $z_j^*$ a new factorization $z_{j, j + 1}$ such that $|z_j^*| < |z_{j, j + 1}|$; consequently, the procedure stops after finitely many iterations by \cite[Proposition~4.5]{fG16}.
	
	The algorithm that we just described yields a sequence of nonnegative integers $K_1, \ldots, K_{k + 1}$, a sequence of factorizations $z_1^*, \ldots, z_{k + 1}^* \in \mathsf{Z}(x)$, and a sequence of finite arithmetic progressions $\sigma_1, \ldots, \sigma_{k + 1}$ with difference $\mathsf{n}(r) - \mathsf{d}(r)$ such that $\sigma_i$ starts at $|z_i^*|$ for each $i \in \llbracket 1, k + 1 \rrbracket$ and contains $|z_{i + 1}^*|$ for each $i \in \llbracket 1, k - 1 \rrbracket$. Moreover, we have $\sigma_i \subseteq \mathsf{L}(x)$ for each $i \in \llbracket 1, k + 1 \rrbracket$. We already showed that $|z_1^*| - |z_{\min}| \leq B_2\cdot (\mathsf{n}(r)^{\delta_0} - \mathsf{d}(r)^{\delta_0})$, where $B_2$ does not depend on $x$ and $z_{\min}$ is the factorization of minimum length of $x$. Now let us suppose that the algorithm stopped at the factorization $z_{k + 1}^*$. Note that if $S = \nn$ then $z_{k + 1}^*$ is the factorization of maximum length of $x$, which implies that $\mathsf{L}(x)$ is an arithmetic sequence with difference $\mathsf{n}(r) - \mathsf{d}(r)$. Since our procedure stopped at $z_{k + 1}^*$, it is not hard to see that 
	\[
		K_{k + 1} \leq (|z_{k + 1}^*| - |z_{k, k + 1}|)(\mathsf{n}(r) - \mathsf{d}(r))^{-1},
	\]
	where $z_{k, k + 1} = \sum_{i = 0}^{n} d_i r^{s_i} \in \mathsf{Z}(x)$, $z_{k + 1}^* = z'' + \sum_{i = m + 1}^{n} d_i r^{s_i} \in \mathsf{Z}(x)$, and $z''$ is the factorization of maximum length of $\pi(\sum_{i = 0}^{m} d_ir^{s_i})$ with $d_i < \mathsf{d}(r)^{\delta_{i - 1}}$ for each $i \in \llbracket 1, m - 1 \rrbracket$ and $d_m < \mathsf{d}(r)^{\delta_{m - 1}} + \mathsf{n}(r)$. Consequently, there exists $B_3 \in \nn$ (which does not depend on $x$) such that $|z_{k + 1}^*| - |z_{k, k + 1}| \leq B_3$. Let $z_{\beta} = \sum_{i = 0}^{n} e_i r^{s_i} \in \mathsf{Z}(x)$ be the factorization we obtain after consecutively applying the identity $\mathsf{d}(r)^{\delta_{\alpha - 1}}r^{s_{\alpha}} = \mathsf{n}(r)^{\delta_{\alpha - 1}}r^{s_{\alpha - 1}}$ with $\alpha \in \llbracket m + 1,n \rrbracket$ as many times as we can (i.e., $K_{k + 1}$) to increase the length of $z_{k + 1}^*$. Evidently, the inequalities 
	\[
		|z_{\beta}| - |z_{k + 1}^*| \leq K_{k + 1}\cdot(\mathsf{n}(r) - \mathsf{d}(r)) \leq B_3
	\]
	hold. Moreover, we have that $e_i < \mathsf{d}(r)^{\delta_{i - 1}}$ for every $i \in \llbracket 1,n \rrbracket$ with $i \neq m$ and 
	\[
		e_m < \mathsf{d}(r)^{\delta_{m - 1}} + \mathsf{n}(r)\cdot K_{k + 1} \leq \mathsf{d}(r)^{\delta_{m - 1}} + \mathsf{n}(r)\cdot B_3\cdot (\mathsf{n}(r) - \mathsf{d}(r))^{-1}.
	\]
	It is easy to see that there exists $B_4 \in \nn$ (which does not depend on $x$ either) such that $|z_{\max}| - |z_{\beta}| \leq B_4$, where $z_{\max} \in \mathsf{Z}(x)$ is the factorization of maximum length of $x$. Hence $|z_{\max}| - |z_{k, k + 1}| \leq 2B_3 + B_4$. Set 
	\[
		B \coloneqq \max(B_2\cdot(\mathsf{n}(r)^{\delta_0} - \mathsf{d}(r)^{\delta_0}), 2B_3 + B_4).
	\]
	Since $|z_{k, k + 1}| \in \sigma_k$, we can conclude that the set $\mathsf{L}(x)$ is an AAP with difference $\mathsf{n}(r) - \mathsf{d}(r)$ and bound $B$ for all $x \in M_{r,S}$.
\end{proof}

\begin{cor} \cite[Theorem~3.3]{CGG19} \label{cor: sets of lengths of rational cyclic semirings}
	Let $M_{r,\nn}$ be a nontrivial atomic rational cyclic semiring. Then $\mathsf{L}(x)$ is an arithmetic progression with difference $|\mathsf{n}(r) - \mathsf{d}(r)|$ for all $x \in M_{r,\nn}$.
\end{cor}

Since the sets of lengths of atomic rational cyclic semirings are arithmetic sequences, it is natural to wonder whether there exist other atomic exponential Puiseux semirings whose sets of lengths are also arithmetic sequences. We now answer this question negatively. 

\begin{cor} \label{cor: delta sets of rational cyclic semirings are singletons}
	Let $M_{r,S}$ be a nontrivial atomic exponential Puiseux semiring. Then $|\Delta(M_{r,S})| = 1$ if and only if $S = \nn$.
\end{cor}

\begin{proof}
	The reverse implication follows from Corollary~\ref{cor: delta sets of rational cyclic semirings}. As for the remaining implication, if $|\Delta(M_{r,S})| = 1$ then, by Proposition~\ref{prop: delta sets of exponential Puiseux semirings}, we have $\delta_i = \delta_j$ for all $i,j \in \nn$, which implies $S = \nn$.
\end{proof}

\begin{remark}
	A monoid $M$, which is not necessarily reduced, is \emph{strongly primary} if $M \neq M^{\times}$ and for every $x \in M \setminus M^\times$ there exists $n \in \nn$ such that $(M \setminus M^{\times})^n \subseteq xM$. Strongly primary monoids play a central role in factorization theory, and they have been widely investigated (see, for instance, \cite{AGFGST2019,AGMR2020}). It is known that sets of lengths of certain strongly primary monoids are AAPs with a common bound~$B$ (\cite[Theorem~4.3.6]{GH06b}). However, nontrivial exponential Puiseux semirings are not strongly primary. In fact, let $M_{r,S}$ be a nontrivial atomic exponential Puiseux semiring, and fix an atom $r^k \in \mathcal{A}(M_{r,S})$ with $k \in S$. Observe that, for any $n \in \nn^{\bullet}$\!, we have that $r^{s_k} \nmid_{M_{r,S}} \sum_{i = 1}^{n} r^{s_{k + i}}$ by Lemma~\ref{lemma: algebraic manipulations for r < 1} and Lemma~\ref{lemma: algebraic manipulations for r > 1}. Therefore, the monoid $M_{r,S}$ is not strongly primary.
\end{remark}

The bounds provided in Theorem~\ref{theorem: sets of lengths} are not tight, and the next example sheds some light upon this observation.

\begin{example} \label{ex: bounds B are not tight}
	Let $r = 3/4$ and $S = \{0, 5, 7\} \cup \nn_{\geq 8}$. By Theorem~\ref{theorem: sets of lengths}, every $L \in \mathcal{L}(M_{r,S})$ is an AAP with difference $1$ and bound $B = 717739$. Indeed, with notation as in \textsc{Case 1} of Theorem 4.1, we have $B_1 = 4^7 - 3^7$. On the other hand, $B_2$ is equal to the number of times we can apply the identity $\mathsf{n}(r)^{\delta_{\alpha}}r^{s_{\alpha}} = \mathsf{d}(r)^{\delta_{\alpha}} r^{s_{\alpha + 1}}$ with $\alpha \in \{0,1\}$ to increase the length of the factorization $z = (3^7 - 1) + (3^2 - 1)(3/4)^5$. Consequently,
	\[
	B_2 = \left\lfloor \frac{(3^7 - 1)}{3^5} \right\rfloor + \left\lfloor \frac{4^5\left\lfloor \frac{(3^7 - 1)}{3^5} \right\rfloor + 3^2 - 1}{9} \right\rfloor = 919.
	\]
	Thus, $B = 717739$. However, it is not hard to see that every $L \in \mathcal{L}(M_{r,S})$ is an AAP with difference $1$ and bound $B' = (4^5 - 3^5) + (4^2 - 3^2) = 788 < 717739$.
\end{example}

Motivated by Example~\ref{ex: bounds B are not tight}, we pose the following question.

\begin{question}
	Given an atomic exponential Puiseux semiring $M_{r,S}$, what is the smallest nonnegative integer $B$ for which $M_{r,S}$ is an AAP with difference $|\mathsf{n}(r) - \mathsf{d}(r)|$ and bound $B$?
\end{question}

We now present \emph{unions of sets of lengths}, a factorization invariant introduced by Chapman and Smith~\cite{ChWWS1990}. For a positive integer $k$, denote by $\,\mathcal{U}_k(M)$ the set of positive integers $m$ for which there exist $a_1, \ldots, a_k, a'_1, \ldots, a'_m \in \mathcal{A}(M)$ such that 
\[
	a _1 + \cdots + a_k = a'_1 + \cdots + a'_m.
\]
It is said that $\,\mathcal{U}_k(M)$ is the \emph{union of sets of lengths} of $M$ containing $k$. 

Next we show that, as it is the case for atomic rational cyclic semirings, unions of sets of lengths of nontrivial atomic exponential Puiseux semirings are arithmetic sequences with difference $|\mathsf{n}(r) - \mathsf{d}(r)|$. 

\begin{prop} \label{prop: unions of sets of lengths of exponential Puiseux semirings}
	Let $M_{r,S}$ be a nontrivial atomic exponential Puiseux semiring. For every $k \in\nn^{\bullet}$\!, the set $\mathcal{U}_k(M_{r,S})$ is an arithmetic sequence with difference $|\mathsf{n}(r) - \mathsf{d}(r)|$.
\end{prop}

\begin{proof}
	Assume that $r < 1$. Fix $k \in \nn^{\bullet}$\!, and let $x \in M_{r,S}$ such that $k \in \mathsf{L}(x)$. Let $m \in \mathsf{L}(x)$. In addition, let $z_{\min} \in \mathsf{Z}(x)$ be the factorization of minimum length of $x$, and let $z, z'' \in \mathsf{Z}(x)$ be factorizations of lengths $m$ and $k$, respectively. By Remark~\ref{remark: any factorization can be transformed into the factorization of minimum length}, there exist factorizations $z_{\min} = z_1, \ldots, z_n = z \in \mathsf{Z}(x)$ such that, for each $i \in \llbracket 1, n - 1 \rrbracket$, we have $|z_{i + 1}| - |z_i| = \mathsf{d}(r)^{\delta_\alpha} - \mathsf{n}(r)^{\delta_\alpha}$ for some $\alpha \in \nn$. Now consider the element $y = r^{F(S) + 1} \cdot x \in M_{r,S}$ and, for each $i \in \llbracket 1,n \rrbracket$, let $z'_i = r^{F(S) + 1}\cdot z_i \in \mathsf{Z}(y)$. Clearly, $|z'_i| = |z_i|$ for every $i \in \llbracket 1,n \rrbracket$. Take $j \in \llbracket 1,n - 1 \rrbracket$, and notice that if $z_j = \sum_{l = 0}^{u} c_l r^{s_l}$ with coefficients $c_0, \ldots, c_u \in \nn$ then $z_{j + 1}$ was obtained from $z_j$ by carrying out the transformation $\mathsf{n}(r)^{\delta_t}r^{s_t} = \mathsf{d}(r)^{\delta_t}r^{s_{t + 1}}$ for some $t \in \llbracket 0,u\rrbracket$. This implies that we can obtain the factorization $z'_{j + 1}$ from $z'_j = \sum_{l = 0}^{u} c_l r^{s_l + F(S) + 1}$ by applying the transformation
	\[
		\mathsf{n}(r)^{\delta_t}r^{s_t + F(S) + 1} = \mathsf{d}(r)^{\delta_t}r^{s_{t + 1} + F(S) + 1},
	\]
	which is valid in this context; in fact, carrying out this transformation once is equivalent to applying $(\mathsf{d}(r)^{\delta_t} - \mathsf{n}(r)^{\delta_t})(\mathsf{d}(r) - \mathsf{n}(r))^{-1}$ times the identity $\mathsf{n}(r)r^n = \mathsf{d}(r)r^{n + 1}$ with $n > F(S)$. In other words, there exist factorizations $z'_j = z^*_1, \ldots, z^*_s = z'_{j + 1} \in \mathsf{Z}(y)$ such that $|z^*_{i + 1}| - |z^*_i| = \mathsf{d}(r) - \mathsf{n}(r)$ for each $i \in \llbracket 1,s-1 \rrbracket$. Hence $\mathsf{L}(y)$ contains a finite arithmetic sequence starting at $|z_1'|$, ending at $|z_n'| = m$, and with difference $\mathsf{d}(r) - \mathsf{n}(r)$. Similarly, $\mathsf{L}(y)$ contains a finite arithmetic sequence starting at $|z_1'|$, ending at $|z''| = k$, and with difference $\mathsf{d}(r) - \mathsf{n}(r)$. Since both arithmetic sequences have an element in common, namely $|z'_1|$, we can think of $\mathcal{U}_k(M_{r,S})$ as the union of infinitely many finite arithmetic sequences with difference $\mathsf{d}(r) - \mathsf{n}(r)$ containing $k$, from which our result follows.
	
	We can use the same idea to tackle the case where $r \in \qq_{>1}\setminus \nn$. We leave the details to the reader.
\end{proof}

Geroldinger and Schmid~\cite{GeSch2019} investigated the intersection of systems of sets of lengths of numerical monoids. In particular, they proved $\cap \, \mathcal{L}(M) = \{\{0\}, \{1\}, \{2\}\}$, where the intersection is taken over all numerical monoids $M \neq \nn$. Gotti~\cite[Corollary~5.7]{fG2019} showed that if we take the previous intersection over all nontrivial atomic Puiseux monoids then we obtain $\cap \, \mathcal{L}(M) = \{\{0\}, \{1\}\}$. We can now describe the intersection of systems of sets of lengths of nontrivial atomic exponential Puiseux semirings $M_{r,S}$ with $|\mathsf{n}(r) - \mathsf{d}(r)|$ fixed. 

\begin{prop}
	For every $m \in \nn^{\bullet}$, we have
	\[
		\bigcap \mathcal{L}(M_{r,S}) = \big\{ \{n\} \mid n \in \nn \big\},
	\]
	where the intersection is taken over all nontrivial atomic exponential Puiseux semirings with $|\mathsf{n}(r) - \mathsf{d}(r)| = m$.
\end{prop}

\begin{proof}
	Fix $m \in \nn^{\bullet}$. Let $M_{r,S}$ be a nontrivial atomic exponential Puiseux semiring with $|\mathsf{n}(r) - \mathsf{d}(r)| = m$, and let $n$ be a nonnegative integer. If $n = 0$ then we have that $\{0\} \in \mathcal{L}(M_{r,S})$ since $\mathsf{L}(0) = \{0\}$, so we may assume that $n > 0$. Now consider the factorization $z_n = \sum_{i=1}^{n} r^{s_i}$, where $s_i \in S$ for each $i \in \llbracket 1,n \rrbracket$. By lemmas~\ref{lemma: algebraic manipulations for r < 1} and \ref{lemma: algebraic manipulations for r > 1}, we have $\mathsf{Z}(\pi(z_n)) = \{z_n\}$ which, in turn, implies that $\{n\} \in \mathcal{L}(M_{r,S})$. Hence the inclusion $\{\{n\} \mid n \in \nn\} \subseteq \cap \,\mathcal{L}(M_{r,S})$ holds. 
	
	Let $p$ be a prime number satisfying that $p > m$, and take $r = (p + m)/p$. For a numerical monoid $S$, note that $M_{r,S}$ is an FFM by \cite[Theorem~5.6]{fG16}. Consequently, all sets of lengths in $\cap\,\mathcal{L}(M_{r,S})$, where the intersection is taken over all nontrivial atomic exponential Puiseux semirings with $|\mathsf{n}(r) - \mathsf{d}(r)| = m$, have finite cardinality. Now consider the rational cyclic semiring $M_{q,\nn}$ with $q = r^{-1}$. By \cite[Proposition~3.7]{SAJBHP2020}, the monoid $M_{q,\nn}$ is atomic. Moreover, for each $x \in M_{q,\nn}$, we have $|\mathsf{L}(x)| \in \{1,\infty\}$. In fact, if $z = \sum_{i = 0}^{n} c_i q^{s_i} \in \mathsf{Z}(x)$ is not the factorization of maximum length of $x$ then $c_j \geq \mathsf{n}(q)$ for some $j \in \llbracket 0,n \rrbracket$ by Lemma~\ref{lemma: algebraic manipulations for r < 1}. By performing the transformation $\mathsf{n}(q)q^{s_j} = \mathsf{d}(q)q^{s_{j} + 1}$, we obtain a factorization $z_1 \in \mathsf{Z}(x)$ such that $|z| < |z_1|$. Since $z_1$ is not the factorization of maximum length of $x$ either, we can repeat this reasoning to obtain $z_2 \in \mathsf{Z}(x)$, which is not the factorization of maximum length of $x$, such that $|z_1| < |z_2|$, and so on. Consequently, we have $|\mathsf{L}(x)| = \infty$, which concludes our argument.
\end{proof}

\section{Omega Primalities and Tame Degrees}

Here we analyze the \emph{omega primality} in the context of nontrivial atomic exponential Puiseux semirings, but first let us introduce some definitions.

\begin{definition} \label{def: omega-primality and tame degree}
	Let $M$ be an atomic monoid, and let $x$ be a nonzero element of $M$.
	\begin{enumerate}
		\item Let $\omega(x)$ denote the smallest $n \in \overline{\nn}$ satisfying that if $x \mid_M a_1 + \cdots + a_m$ for some $a_1, \ldots, a_m \in\mathcal{A}(M)$ then $x \mid_M a_{i_1} + \cdots + a_{i_k}$, where $\{i_1, \ldots, i_k\} \subseteq \llbracket 1,m \rrbracket$ and $k \leq n$. In addition, set
		\[
		\omega(M) \coloneqq \sup \{\omega(a) \mid a \in\mathcal{A}(M)\}.
		\]
		The elements $\omega(x)$ and $\omega(M)$ are called the \emph{omega primalities} of $x$ and $M$\!, respectively.
		
		\item For $a \in \mathcal{A}(M)$, let $\mathsf{t}(a)$ denote the smallest $n \in \overline{\nn}$ satisfying that if $\mathsf{Z}(y) \cap (a + \mathsf{Z}(M)) \neq \emptyset$ for some $y \in M$ and $z \in \mathsf{Z}(y)$ then there exists $z' \in \mathsf{Z}(y) \cap (a + \mathsf{Z}(M))$ such that $\mathsf{d}(z,z') \leq n$. In addition, set
		\[
		\mathsf{t}(M) \coloneqq \sup \{\mathsf{t}(a) \mid a \in\mathcal{A}(M)\}.
		\]
		The elements $\mathsf{t}(a)$ and $\mathsf{t}(M)$ are called the \emph{tame degrees} of $a$ and $M$\!, respectively.
	\end{enumerate}
\end{definition} 

With notation as in Definition~\ref{def: omega-primality and tame degree}, the omega function $\omega\colon M \to \overline{\nn}$ was introduced by Geroldinger~\cite{aG1997}, and it measures how far is a nonzero element from being prime. Note that $x \in M^{\bullet}$ is prime if and only if $\omega(x) = 1$. Numerous papers have been dedicated to study the computational aspects of this factorization invariant (e.g., \cite{AnChKa2011,GGMFVT2014}).

Given a nontrivial atomic exponential Puiseux semiring $M_{r,S}$, we have $\rho(M_{r,S}) = \infty$ by \cite[Theorem~3.2]{FGCO}, which implies $\omega(M_{r,S}) = \infty$ and $\mathsf{t}(M_{r,S}) = \infty$ by \cite[Proposition~3.6 and Proposition 3.5]{GeKa2010}. Next we offer several characterizations of nontrivial atomic exponential Puiseux semirings with finite omega functions, but first we collect some technical results and introduce an additional definition.

\begin{lemma} \label{lemma: x dividing the set of smaller atoms}
	Let $M_{r,S}$ be a nontrivial atomic exponential Puiseux semiring with $r > 1$. If $x \in M_{r,S}$ and $k \in \nn$ such that $r^{s_k} \mid_{M_{r,S}} x$ then, for every $z = \sum_{i = 0}^{n} c_i r^{s_i} \in \mathsf{Z}(x)$ with coefficients $c_0, \ldots, c_n \in \nn$, we have that either
	\[
		r^{s_k} \,\bigg|_{M_{r,S}}\, \pi\left(\sum_{i = 0}^{k} c_i r^{s_i}\right) \hspace{0.6cm}\text{or}\hspace{0.6cm} r^{s_k} \,\bigg|_{M_{r,S}}\, \pi\left(\sum_{i = k + 1}^{n} c_i r^{s_i}\right)\!.  
	\]
\end{lemma}

\begin{proof}
	Let $z' = \sum_{i = 0}^{m} d_i r^{s_i} \in \mathsf{Z}(x)$ with coefficients $d_0, \ldots, d_m \in \nn$ such that $m \geq k$ and $d_k \neq 0$. Clearly, such a factorization $z'$ exists given that $r^{s_k} \mid_{M_{r,S}} x$. By Remark~\ref{remark: any factorization can be transformed into the factorization of minimum length}, there exist factorizations 
	\[
		z' = z_1, \ldots, z_l = z_{\max} = \sum_{i = 0}^{u} e_i r^{s_i} \in \mathsf{Z}(x)
	\]
	such that $z_{\max}$ is the factorization of maximum length of $x$ and, for each $i \in \llbracket 1, l - 1 \rrbracket$, we have $|z_{i + 1}| - |z_i| = |\mathsf{n}(r)^{\delta_t} - \mathsf{d}(r)^{\delta_t}|$ for some $t \in \nn$. Since $d_k \neq 0$, it is not hard to see that $r^{s_k} \mid_{M_{r,S}} \pi(\sum_{i = 0}^{k} e_ir^{s_i})$. Now assume that $r^{s_k} \nmid_{M_{r,S}} \pi(\sum_{i = k + 1}^{n} c_i r^{s_i})$, and let $z''$ be the factorization of maximum length of $\pi(\sum_{i = k + 1}^{n} c_i r^{s_i})$. Observe that if an atom $r^{s_t}$ shows up in the factorization $z''$ then $t > k$; otherwise, we can repeatedly apply the identity $\mathsf{d}(r)^{\delta_{\alpha}} r^{s_{\alpha + 1}} = \mathsf{n}(r)^{\delta_{\alpha}} r^{s_{\alpha}} $ to transform the factorization $\sum_{i = k + 1}^{n} c_i r^{s_i}$ into $z''$ as described in Remark~\ref{remark: any factorization can be transformed into the factorization of minimum length}, but since $z''$ would contain atoms smaller than $r^{s_{k + 1}}$\!, at some point we would have to apply the transformation $\mathsf{d}(r)^{\delta_{k}} r^{s_{k + 1}} = \mathsf{n}(r)^{\delta_k}r^{s_k}$\!, a contradiction. Consequently, we have 
	\[
		\pi\left(\sum_{i = 0}^{k} c_i r^{s_i}\right) = \pi\left(\sum_{i = 0}^{k} e_ir^{s_i}\right),
	\]
	from which our argument follows.
\end{proof}

\begin{lemma}\label{lemma: elements of truncated rational cyclic semirings have finite omega function}
Let $M_{r,S}$ be a nontrivial atomic exponential Puiseux semiring with $r > 1$. Then $\omega(r^s) < \infty$ for $s > F(S)$.
\end{lemma}

\begin{proof}
	Let $m \in \nn$ such that $s_m \geq F(S) + 1$, and set $K \coloneqq \max(\mathsf{d}(r), \sum_{i = 0}^{m - 1} \mathsf{n}(r)^{s_m - s_i})$. We shall prove that $\omega(r^{s_m}) \leq K$. Let $x \in M_{r,S}$ such that $r^{s_m} \,|_{M_{r,S}}\, x$, and consider a factorization $z = \sum_{i = 0}^{n} c_ir^{s_i} \in \mathsf{Z}(x)$ with coefficients $c_0, \ldots, c_n \in \nn$. There is no loss in assuming that $m < n$ and $c_m = 0$. By Lemma~\ref{lemma: x dividing the set of smaller atoms}, we have that either 
	\[
		r^{s_m} \,\bigg|_{M_{r,S}}\, \pi\left(\sum_{i = 0}^{m} c_ir^{s_i}\right) \hspace{.6cm} \text{ or } \hspace{.6 cm} r^{s_m} \,\bigg|_{M_{r,S}}\, \pi\left(\sum_{i = m + 1}^{n} c_ir^{s_i}\right).
	\]
	In the latter case, there exists $j \in \llbracket m + 1,n \rrbracket$ such that $c_j \geq \mathsf{d}(r)$ by Lemma~\ref{lemma: algebraic manipulations for r > 1} and, obviously, $r^{s_m} \,|_{M_{r,S}}\, \mathsf{d}(r)r^{s_j}$\!. In other words, $r^{s_m}$ divides a sub-factorization of $z$ with length at most $K$. Now suppose that $r^{s_m} \,|_{M_{r,S}}\, \pi(\sum_{i = 0}^{m} c_ir^{s_i})$. Without loss of generality, we can assume that $K < \sum_{i = 0}^{m - 1} c_i$, so $m \neq 0$. Consequently, there exists $j \in \llbracket 0,m - 1 \rrbracket$ such that $c_j \geq \mathsf{n}(r)^{s_m - s_j}$\!, but $r^{s_m} \,|_{M_{r,S}}\, \mathsf{n}(r)^{s_m - s_j}r^{s_j}$\!. Therefore, the inequalities $\omega(r^{s_m}) \leq K < \infty$ hold.
\end{proof}

\begin{definition} \label{def: ACCP-presentable}
	Let $M$ be an atomic monoid and $\sigma$ a binary relation on $\mathsf{Z}(M)$. 
	\begin{enumerate}
		\item Let $G_{\sigma}$ be the graph whose vertices are the $\mathcal{R}$-classes of $\mathsf{Z}(M)$ and whose edges are the pairs $(R, R')$ satisfying that $R \neq R'$ and $(z,z') \in \sigma \cup \sigma^{-1}$ for some $z \in R$ and $z' \in R'$.
		
		\item We say that $M$ is \emph{ACCP-presentable} provided that $\sigma' \subseteq \,\sim_M$ is a presentation of $M$ if and only if, for every $x \in M$, all the $\mathcal{R}$-classes in $\pi^{-1}(x)$ are in the same connected component of $G_{\sigma'}$.
	\end{enumerate}
\end{definition}

\begin{remark}
	Let $M_{r,S}$ be a nontrivial atomic exponential Puiseux semiring, and let $\sigma \subseteq \, \sim_{M_{r,S}}$. Using Proposition~\ref{prop:congruence generated description} and Proposition~\ref{prop: betti elements}, it is easy to prove that if $\sigma$ is a presentation of $M_{r,S}$ then, for every $x \in M_{r,S}$, all the $\mathcal{R}$-classes in $\pi^{-1}(x)$ are in the same connected component of $G_{\sigma}$; we leave the details to the reader. Consequently, proving that $M_{r,S}$ is ACCP-presentable amounts to verify that, for every $\sigma' \subseteq \sim_{M_{r,S}}$, if all the $\mathcal{R}$-classes in $\pi^{-1}(x)$ are in the same connected component of $G_{\sigma'}$ for every $x \in M_{r,S}$ then $\sigma'$ is a presentation of $M_{r,S}$.
\end{remark}

It is well known that a monoid satisfying the ACCP is ACCP-presentable (\cite[Theorem~1]{BG12}). Next we show that these two properties characterize the atomic exponential Puiseux semirings with finite omega functions.

\begin{theorem}\label{th: finite local primality} 
	Let $M_{r,S}$ be an atomic exponential Puiseux semiring. The following statements are equivalent.
	\begin{enumerate}
		\item $r \geq 1$.
		\item $M_{r,S}$ satisfies the ACCP.
		\item $M_{r,S}$ is ACCP-presentable.
		\item $\omega(a) < \infty$ for all $a \in \mathcal{A}(M_{r,S})$.
	\end{enumerate}
\end{theorem}

\begin{proof}
	The statements $(1)$ and $(2)$ are equivalent by virtue of \cite[Proposition~4.5]{fG16} and \cite[Corollary~4.2]{SAJBHP2020}, while $(4)$ implies $(2)$ follows from \cite[Lemma~3.3]{GH08}. 
	
	Now we proceed to argue that $(1)$ implies $(4)$. If $r \in \nn$ then this implication trivially holds. Consequently, we may assume that $M_{r,S}$ is nontrivial. Let $m \in \nn$ such that $s_m = F(S) + 1$, and fix $n \in \nn$. By Lemma~\ref{lemma: elements of truncated rational cyclic semirings have finite omega function}, the inequality $\omega(r^{s_k}) < \infty$ holds for each $k \geq m$; consequently, we may assume that $n < m$. Then $r^{s_n} \,|_{M_{r,S}}\, \mathsf{d}(r)^{s_m - s_n}r^{s_m}$ but, for every positive integer $k$, we have that $\omega(k\cdot r^{s_m}) \leq k\cdot \omega(r^{s_m}) < \infty$ by \cite[Lemma~3.3]{GH08} and Lemma~\ref{lemma: elements of truncated rational cyclic semirings have finite omega function}. Applying \cite[Lemma~3.3]{GH08} again, we obtain that $\omega(r^{s_n}) < \infty$, and our argument concludes.
	
	As we mentioned above, a monoid satisfying the ACCP is ACCP-presentable by \cite[Theorem~1]{BG12}. Next we prove that $M_{r,S}$ is not ACCP-presentable when $r < 1$. We can write $\mathsf{d}(r)$ as $\mathsf{d}(r) = q\cdot\mathsf{n}(r) + c$ with $q \in \nn$ and $1 \leq c < \mathsf{n}(r)$. Since $r < 1$, the inequality $q \geq 1$ holds. As before, let $m \in \nn$ such that $s_m = F(S) + 1$. Now let
	\[
		\rho = \left\{\big(\mathsf{n}(r)r^n\!, \,\mathsf{d}(r)r^{n + 2} + ((q - 1)\mathsf{n}(r) + c)r^{n + 1}\big) \,\big|\, n \in \nn_{>F(S)}\right\} \subseteq \mathsf{Z}(M_{r,S}) \times \mathsf{Z}(M_{r,S}),
	\]
	where $F(S)$ represents the Frobenius number of the numerical monoid $S$\!, and let
	\[
		\gamma = \left\{\left(\mathsf{n}(r)^{\delta_n}r^{s_n}\!, \mathsf{d}(r)^{\delta_n}r^{s_{n + 1}}\right) \,\big|\, n \in \nn_{<m}\right\}\subseteq \mathsf{Z}(M_{r,S}) \times \mathsf{Z}(M_{r,S}).
	\]
	Take $\sigma = \rho \cup \gamma$. Obviously, the inclusion $\sigma \subseteq \,\sim_{M_{r,S}}$ holds. Note that, for each $x \in M_{r,S}$, all the $\mathcal{R}$-classes in $\pi^{-1}(x)$ are in the same connected component of $G_{\sigma}$ by Proposition~\ref{prop: betti elements}. Let $k \in \nn$ such that $k \geq s_m$ and $\mathsf{n}(r)r^k < r^{s_n}$ for each $n \in \llbracket 0, m - 1 \rrbracket$. Clearly, if $(\mathsf{n}(r)r^k, \mathsf{d}(r)r^{k + 1}) \in \sigma^\sharp$ then $(\mathsf{n}(r)r^k, \mathsf{d}(r)r^{k + 1}) \in \rho^\sharp$. However, it is easy to see that if $(z,z') \in \rho$ then $\big||z| - |z'|\big| = 2\cdot(\mathsf{d}(r) - \mathsf{n}(r))$, and the same is true for any pair in $\rho^t$; consequently, if $(z, z') \in (\rho^t)^e$ then $\big||z| - |z'|\big| = 2s(\mathsf{d}(r) - \mathsf{n}(r))$ for some $s \in \nn$. Hence $(\mathsf{n}(r)r^k, \mathsf{d}(r)r^{k + 1}) \not \in \rho^\sharp$. Therefore, $\sigma$ is not a presentation of $M_{r,S}$ which, in turn, implies that $M_{r,S}$ is not ACCP-presentable.
\end{proof}

An atomic monoid $M$ is called \emph{locally} (resp. \emph{globally}) \emph{tame} provided that $\mathsf{t}(a) < \infty$ (resp. $\mathsf{t}(M) < \infty$) for every $a \in \mathcal{A}(M)$. We can now describe the atomic exponential Puiseux semirings that are locally (or globally) tame, but first we need to introduce a definition.

\begin{definition} 
	Let $M$ be an atomic monoid, and let $x$ be an element of $M$. For $n \in \nn^{\bullet}$\!, we denote by $\mathsf{Z}_{\min}(n,x)$ the set of all factorizations $z \in \mathsf{Z}(M)$ with length at most $n$ satisfying that the smallest sub-factorization of $z$ divisible by $x$ in $M$ is precisely $z$. In addition, set
	\[
	\tau(x) \coloneqq \sup_n \sup_z\left\{ \min\mathsf{L}\left( \pi(z) - x \right) \mid z \in \mathsf{Z}_{\min}(n,x) \right\}\!.
	\]
\end{definition}

\begin{cor} \label{cor: nontrivial exponential Puiseux semirings are not locally tame}
	Any nontrivial atomic exponential Puiseux monoid is neither locally tame nor globally tame.
\end{cor}

\begin{proof}
	Let $M_{r,S}$ be a nontrivial atomic exponential Puiseux semiring. By virtue of \cite[Theorem~3.2]{FGCO}, we have $\rho(M_{r,S}) = \infty$, which implies that $M_{r,S}$ is not globally tame by \cite[Theorem~1.6.6]{GH06b}. Now if $r < 1$ then $M_{r,S}$ is not locally tame by Theorem~\ref{th: finite local primality} and \cite[Theorem~3.6]{GH08}. On the other hand, if $r > 1$ then it is straightforward to adapt the argument used in the proof of \cite[Theorem~5.6]{CGG19} to show that the equality $\tau(r^{F(S) + 1}) = \infty$ holds, so we leave this task to the reader. Our result then follows from \cite[Theorem~3.6]{GH08}. 
\end{proof}

The omega function of a Puiseux monoid satisfying the ACCP is not, in general, a finite function as the following example illustrates.

\begin{example}
	Let $M = \langle \frac{p - 1}{p} \mid p \in \mathbb{P} \rangle$. It is easy to see that $\mathcal{A}(M) = \{\frac{p - 1}{p} \mid p \in \mathbb{P}\}$, which implies that $M$ is atomic. In fact, the monoid $M$ is a BFM by \cite[Proposition~4.5]{fG16}. Let us fix $p \in \mathbb{P}$. A straightforward computation shows that, for $q \in \mathbb{P}\setminus\{p\}$ and $n \in \nn$, if $(p - 1)/p \,\,|_{M}\,\, n(q - 1)/q$ then the inequality $n \geq q$ holds. Consequently, $\omega((p - 1)/p) \geq q$ for all $q \in \mathbb{P}\setminus\{p\}$. In other words, we have $\omega((p - 1)/p) = \infty$ for each $p \in \mathbb{P}$.
\end{example}

\section{Acknowledgments}

The author wants to thank Felix Gotti and an anonymous referee for valuable feedback that helped improve the quality of this manuscript. While working on the same, the author was generously supported by the CAM Summer Research Fellowship.

\bigskip

\end{document}